\documentclass[12pt]{article}
\usepackage{amsmath,amsthm,amssymb}
\allowdisplaybreaks
\usepackage{color}

\newcommand{\ds}{\displaystyle}

\newcommand{\Om}{\Omega}

\newcommand{\la}{\lambda}

\newcommand{\va}{\varepsilon}

\newcommand{\dis}{\displaystyle}

\theoremstyle{plain}
\newtheorem{theorem}{Theorem}[section]

\newtheorem{lemma}{Lemma}[section]
\newtheorem{proposition}{Proposition}[section]

\numberwithin{equation}{section}

\newcommand{\R}{{\mathbb R}}

\begin{document}
	\title{The Br\'ezis-Nirenberg Result for the Fractional Elliptic Problem with Singular Potential
		\thanks {The research was supported by the Natural Science Foundation of China
			(11101160,11271141) and the China Scholarship Council (201508440330)  }}
	\author{ Lingyu Jin, Lang Li and Shaomei Fang\\\small{Department of Mathematics, South China Agricultural
			University,}
		\\\small{ Guangzhou 510642, P. R. China}
		\\
	}
	\date{}
	\maketitle
	\begin{abstract}
		In this paper, we are
		concerned with the following type of fractional problems:
		$$
		\begin{cases}\dis
		(-\Delta)^{s} u-\mu\frac{u}{|x|^{2s}}-\lambda  u=|u|^{2^*_{s}-2}u+f(x,u), &\text{ in } \Omega,\ \
		\,\\
		u=0\,&\text{ in } \R^N\backslash\Omega
		\end{cases}
		\eqno {(*)} $$
		where  $s\in (0,1)$,   $2^*_{s}=2N/(N-2s)$ is the critical Sobolev exponent, $f(x,u)$ is a lower order perturbation of critical Sobolev nonlinearity. We obtain the existence of the solution for (*) through variational methods. In particular we derive a Br\'ezis-Nirenberg type result when $f(x,u)=0$. 
		 
	\end{abstract}
	
	\bigbreak
	\bigbreak
	\bigbreak

	\section{Introduction}
	Problems involving critical Sobolev nonlinearity have been an interesting topic for a long time.
	In the celebrated paper \cite{BN}, Br\'ezis and Nirenberg  investigated the problem
	\begin{align}\label{1.1j}
	\begin{cases}
	-\Delta u-\lambda u =|u|^{2^*-2}u, &\text{ in } \Omega,\ \
	\,\\
	u=0\,&\text{ on } \partial\Omega 
	\end{cases}\end{align}
	where $\Omega \subset \R^N$ is a bounded domain with smooth boundary, $N\geq 3$, $2^*=\frac{2N}{N-2}$. They proved the following existence and nonexistence results.
	 
	 Let $\lambda_1$ is the first eigenvalue of $-\Delta $ with homogeneous Dirichlet boundary conditions. Then,

	A) if $N\geq 4$, problem (\ref{1.1j}) has a positive solution for $0<\lambda <\lambda_1$;
	
	B) if $N=3$, there exists a constant $\lambda* \in (0, \lambda_1)$ such that problem (1.1) has a positive solution for $\lambda \in (\lambda*,\lambda_1)$;
	
	C) If $N=3$ and $\Omega $ is a ball, then $\lambda*=\frac{1}{4}\lambda_1$, and problem (\ref{1.1j}) has no solution for $\lambda \in (0,\lambda*)$.
	
	Also, Jannelli in \cite{J} obtained the Br\'ezis-Nirenberg existence and nonexistence results for a class of problem with the critical Sobolev nonlinerity and the Hardy term as follows
	\begin{equation}
	\label{1.20}
	\begin{split}
	\begin{cases}\dis
	-\Delta u-\mu\frac{u}{|x|^{2}}-\lambda  u=|u|^{2^*-2}u &\text{ in } \Omega,\ \
	\,\\
	u=0\,&\text{ on } \partial\Omega.
	\end{cases}
	\end{split}
	\end{equation}
	For more results on  (\ref{1.20}), please see \cite{ CP,CP1,Ch,DJP,J,JD} and the references therein. 
 
 Recently  Servadei and Valdinoci \cite{SV} extended the results of \cite{BN}(part A) to the following nonlocal fractional equation
		\begin{align}\label{1.1}
		\begin{cases}
		(-\Delta)^{s} u-\lambda u =|u|^{2^*_{s}-2}u, &\text{ in } \Omega,\ \
		\,\\
		u=0\,&\text{ in } \R^N\backslash\Omega 
		\end{cases}\end{align}
		where $s\in (0,1)$, $(-\Delta)^s$ is the fractional operator, which may be defined 
		\begin{equation}\label{1.4jj}
		(-\Delta )^s u=2c_{N,s}P.V.\int_{\R^N}\frac{u(x)-u(y)}{|x-y|^{N+2s}}dy, \,x\in \R^N
		\end{equation}
	with $c_{N,s}=2^{2s-1}\pi^{-\frac{N}{2}}\frac{\Gamma(\frac{N+2s}{2})}{|\Gamma(-s)|}$	(see \cite{BV}).
		
Motivated by the above papers, in this paper we consider the following nonlocal fractional equation with singular potential
		\begin{align}
		\label{1.4}
		\begin{cases}\dis
		(-\Delta)^{s} u-\mu\frac{u}{|x|^{2s}}-\lambda  u=|u|^{2^*_{s}-2}u+f(x,u), &\text{ in } \Omega,\ \
		\,\\
		u=0\,&\text{ in } \R^N\backslash\Omega
		\end{cases}
	\end{align}
where	$0\in \Omega\subset \R^N$  is a bounded domain with smooth boundary, $N\geq 3$, $0<s<1$,  $(-\Delta)^s$ is the fractional operator defined in (\ref{1.4jj}),
 $2^*=2N/(N-2s)$ is the critical Sobolev exponent, and $f(x,u)$ is a lower order perturbation of critical Sobolev nonlinearity.  Problem (\ref{1.4}) can be considered as doubly critical due to the critical power $|u|^{2^*_{s}-2}$ in the semilinear term and the spectral anomaly of the Hardy potential $\dis\frac{u}{|x|^{2s}}$.
 The fractional framework introduces nontrivial difficulties which have interest in themselves. In this paper
we investigate the Br\'ezis-Nirenberg type result for (\ref{1.4}). Firstly we
   prove the existence of solutions for (\ref{1.4}) through variational methods under some natural assumptions on $f(x,u)$. Secondly we derive a Br\'ezis-Nirenberg type result for  the case $f(x,u)=0$. For other existence results involving (\ref{1.4}), you can refer to \cite{DMP, GS, WY}.
   
Notice  that equation (\ref{1.4}) has a variational structure. The variational functional of (\ref{1.4}) is 
\begin{equation}
\begin{aligned}\label{1.6}
J_{\lambda}(u)&=\frac{1}{2}\Bigl (c_{N,s}\int_{\R^{2N}}\frac{|u(x)-u(y)|^2}{|x-y|^{N+2s}}dxdy-\lambda\int_{\Omega}|u(x)|^2dx-\mu\int_{\Omega}\frac{|u(x)|^2}{|x|^{2s}}dx\Bigl )\\&\,\,-\frac{1}{2^*_s}\int_{\Omega}|u(x)|^{2^*_s}dx-\int_{\Omega}F(x,u)dx,
\end{aligned}
\end{equation}
where 
\begin{equation}
F(x,t)=\int^t_0f(x,\tau)d\tau.
\end{equation}
We say $u(x)$ is a weak solution of (\ref{1.4}) if $u(x)$ satisfy that
\begin{equation}
\begin{aligned}\label{1.5}
&c_{N,s}\int_{\R^{2N}}\frac{(u(x)-u(y))(\phi(x)-\phi(y))}{|x-y|^{N+2s}}dxdy-\mu\int_{\Om}\frac{u(x)\phi(x)}{|x|^{2s}}dx-\lambda\int_{\Omega}u(x)\phi(x)dx
\\&=\int_{\Omega}|u(x)|^{2^*_s-2}u(x)\phi(x)dx+\int_{\Omega} f(x,u)\phi dx,\,\, \phi \in H^s(\R^N) \text{ with } \phi=0, \text { a.e.  in } \R^N\backslash \Omega,\\
&u\in H^s (\R^N) \text{ with } u=0 \text{ a.e. in } \R^N\backslash\Omega.
\end{aligned}
\end{equation}
As a result of  Hardy inequality  \cite{DMP}, it is easy to see that
 \begin{equation}\label{1.4j}
  \Gamma_{N,s}\big (\int_{\R^N}\frac{|u(x)|^{2}}{|x|^{2s}}dx\bigl )\leq  c_{N,s}\int_{\R^N}\frac{|u(x)-u(y)|^2}{|x-y|^{N+2s}}dy, \forall  u\in C_0^\infty (\R^N)，
 \end{equation}
 where \begin{equation}
 \label{1.8j}\Gamma_{N,s}=2^{2s}\frac{\Gamma^2(\frac{N+2s}{4})}{\Gamma^2(\frac{N-2s}{4})}, c_{N,s}=2^{2s-1}\pi^{-\frac{N}{2}}\frac{\Gamma(\frac{N+2s}{2})}{|\Gamma(-s)|}, \end{equation}
  for $0<\mu<\Gamma_{N,s}$, we can define the first eigenvalue $\lambda_{1,\mu}$ of the following problem
	\begin{align}
	\label{1.10j}
	\begin{cases}\dis
	(-\Delta)^{s} u-\mu\frac{u}{|x|^{2s}}=\lambda u   &\text{ in } \Omega,\ \
	\,\\
	u=0\,&\text{ in } \R^N\backslash\Omega
	\end{cases}
	\end{align}
 as in (\ref{2.1}).
In this paper, we consider problem (\ref{1.4}) in the case  $0<\lambda<\lambda_{1,\mu}$. As in the classical case of Laplacian, 
the main difficulty for nonlocal elliptic problems with critical nonlinearity is the lack of compactness. This would cause the functional $J_\lambda$  not to satisfy the Palais-Smale condition.  
Nevertheless,  we are able to use the Mountain-Pass Theorem without the Palais-Smale condition (as given in \cite{BN}, Theorem 2.2) to overcome the lack compactness of Sobolev embedding. 

	Before introducing our main results, we give  some
	notations and assumptions.
	
	{\bf Notations and assumptions:}
	
	Denote $c$ and $C$ as
	arbitrary constants. Let $B_r(x)$ denote a ball centered at $x$ with
	radius $r$, $B_r$ denote a ball centered at 0 with radius $r$ and  $B_r^C=\R^N\setminus B_r$. Let $o_n(1)$ be an infinitely small quantity such that $o_n(1)\rightarrow 0$ as $n\rightarrow \infty$.
	
Let $f(x,u)$ in  (\ref{1.4}) be a Carath\'eodary function $f:\Omega\times\R\rightarrow \R$  satisfying the following conditions:\begin{eqnarray}
f(x,u) \text{ is continuous with respect to } u;\label{1.9}\\
\lim_{u\rightarrow 0} \frac{f(x,u)}{u}=0 \text{ uniformly in } x\in \Omega;\label{1.10}\\
\lim_{u \rightarrow \infty} \frac{f(x,u)}{|u|^{2^*_s-1}}=0 \text{ uniformly in } x\in \Omega.\label{1.11}
\end{eqnarray}
(\ref{1.9})-(\ref{1.11}) ensure that $f(x,u)$ is a lower order perturbation of the critical nonlinearity.
From Lemma 5 and Lemma 6 in \cite{SV}, for any $\va>0$ there exist constants $M(\va)>0,\delta(\va)>0$ such that for any $u\in \R,x\in \Omega$,
\begin{eqnarray}\label{1.12j}
\begin{split}
|f(x,u)|\leq \va |u|^{2^*-1}+M(\va);\,\,
|f(x,u)|\leq \delta(\va) |u|^{2^*-1}+\va |u|.
\end{split}
\end{eqnarray}
We denote by $H^s(\R^N)$ the usual fractional Sobolev space endowed with the norm 
\begin{equation}\label{1.13j}
\|u\|_{H^s(\R^N)}=\|u\|_{L^2(\R^N)}+(c_{N,s}\int_{\R^{2N}}\frac{|u(x)-u(y)|^2}{|x-y|^{N+2s}}dxdy)^{1/2}.
\end{equation}

Let $X_0$ be the functional space defined as
\begin{equation}
X_0=\{u\in H^s(\R^N), \,\, u=0 \text{ a.e. in } \R^N\backslash \Omega\}
\end{equation}
 with the norm 
$$ \dis\|u\|_{X_\mu}=\Bigl(c_{N,s}\int_{\R^{2N}}\frac{|u(x)-u(y)|^2}{|x-y|^{N+2s}}dxdy-\mu\int_{\Omega}\frac{|u(x)|^{2}}{|x|^{2s}}dx\Bigl)^{1/2},$$
which is equivalent to the norm $\|\cdot\|_{H^s(\R^N)}$ for $0<\mu<\Gamma_{N,s}$ (see in Lemma \ref{l2.2}).
Define 
\begin{equation}\label{1.15}\dis
\ds S=\inf_{u\in H^s(\R^N)\backslash \{0\}}  \frac{c_{N,s}\int_{\R^{2N}}\frac{|u(x)-u(y)|^2}{|x-y|^{N+2s}}dxdy-\mu\int_{\R^N}\frac{|u(x)|^{2}}{|x|^{2s}}dx}{(\int_{\R^N}|u(x)|^{2^*_s}dx)^{2/2^*_s}},
\end{equation}	
the Euler equation associated to (\ref{1.15}) is as follows
\begin{equation}
	\label{1.16}
	(-\Delta)^{s} u-\mu\frac{u}{|x|^{2s}}=|u|^{2^*_{s}-2}u  \text{ in } \R^N.
\end{equation}
In particular it has been showed  	in Theorem 1.2 of \cite{DMP} that for any solution $u(x) \in H^s(\R^N)$ of (\ref{1.16}), there exist two positive constants $c, C$ such that
\begin{equation}\label{1.17}
	\frac{c}{\Bigl(|x|^{1-\eta_\mu}(1+|x|^{2\eta_\mu})\Bigl )^{\frac{N-2s}{2}}}\leq u(x)\leq \frac{C}{\Bigl (|x|^{1-\eta_\mu}(1+|x|^{2\eta_\mu})\Bigl )^{\frac{N-2s}{2}}}, \text{ in } \R^N\backslash \{0\}
\end{equation}
where \begin{equation}\label{1.19}
\eta_\mu =1-\frac{2\alpha_\mu}{N-2s}, 
\end{equation}
and $\alpha_\mu\in (0,\frac{N-2s}{2})$ is a suitable parameter whose explicit value will be determined as the unique solution to the following equation
\begin{equation}\label{1.20jjj}
\varphi_{s,N}(\alpha_\mu)=2^{2s} \frac{\Gamma(\frac{\alpha_\mu+2s}{2})\Gamma(\frac{N-\alpha_\mu}{2})}{\Gamma(\frac{N-\alpha_\mu-2s}{2})\Gamma(\frac{\alpha_\mu}{2})}=\mu
\end{equation}
and $\varphi$ is strictly increasing. Obviously any achieve function  $U(x)$ of $S$ also satisfies (\ref{1.17}).

Define the constant 
\begin{equation}\label{1.20j}
S_{\lambda}=\inf_{u\in X_0\backslash\{0\}} \frac{c_{N,s}\int_{\R^{2N}}\frac{|u(x)-u(y)|^2}{|x-y|^{N+2s}}dxdy-\mu\int_{\Omega}\frac{|u(x)|^{2}}{|x|^{2s}}dx
	-\lambda\int_{\Omega}|u(x)|^2dx}{(\int_{\Omega}|u(x)|^{2^*_s}dx)^{\frac{2}{2^*_s}}}，
\end{equation}
Obviously, $S_\lambda \leq S$ because $\lambda>0$. Just as in \cite{BN} and \cite{J}, we will actually focus on the case when the strict inequality occurs. The first results of the present paper is the following one:
\begin{theorem}\label{a}
	Let $s\in (0,1), N>2s, \Omega$ be an open bounded set of $\R^N$,  $\lambda_{1,\mu}$ is the first eigenvalue of (\ref{1.10j}), $0<\mu<\Gamma_{N,s}$ ($\Gamma_{N,s}$ is defined in (\ref{1.8j})) and $0<\lambda<\lambda_{1,\mu}$. Let $f$ be a Carath\'eodory function verifying (\ref{1.9})-(\ref{1.11}), 
	assume that 
	there exists $u_0\in H^s(\R^N)\backslash \{0\}$ with $u_0\geq 0$ a.e. in $\Omega$ such that
	\begin{equation}\label{1.20jj}
	\sup_{t\geq 0}J_{\lambda}(tu_0)<\frac{s}{N}S^{\frac{N}{2s}},
	\end{equation}
	 problem (\ref{1.4}) admits a solution $u\in H^s(\R^N)$, which is not identically zero, such that $u=0$ a.e. in $\R^N\backslash\Omega$.
\end{theorem}
If assumption (\ref{1.20jj}) is satisfied, then it follows Theorem \ref{b}. 
\begin{theorem}\label{b}
	Let $s\in (0,1), N>2s, \Omega$ be an open bounded set of $\R^N$, $\lambda_{1,\mu}$ is the first eigenvalue of (\ref{1.10j}), and $\mu\leq \varphi^{-1}(\frac{N-4s}{2})$ ($\varphi$ is defined in (\ref{1.20jjj})). For any $0<\lambda<\lambda_{1,\mu}$ and  $f(x,u)\equiv0$,  problem (\ref{1.4}) admits a solution $u\in H^s(\R^N)$, which is not identically zero, such that $u=0$ a.e. in $\R^N\backslash\Omega$.
\end{theorem}

   Recall the results in \cite{J} as the follows.
   
    Let $\bar \lambda_{1}$ denote the first eigenvalue of the operator $-\Delta -\frac{1}{|x|^2}$ in $\Omega$ with Dirichlet boundary conditions and  $\bar{\mu}=(\frac{N-2}{2})^2$. 
    
   a) If $\mu\leq \bar \mu-1$, problem (\ref{1.20}) has a positive solution for $0<\lambda <\bar\lambda_{1}$;
   
   b) if $\bar{\mu}-1<\mu<\bar\mu $, there exists a constant $\lambda^* \in (0, \bar \lambda_1)$ such that problem (\ref{1.20}) has a positive solution for $\lambda \in (\lambda^*,\bar\lambda_{1})$;
   
   c) If $\bar{\mu}-1<\mu<\bar\mu $ and $\Omega $ is a ball,  then problem (\ref{1.20}) has no solutions for $\lambda\leq \lambda^*$.
    	
    As for the results A), B),  C) for problem (\ref{1.1}), the space dimension $N$ plays a fundamental role when one seeks solutions of (\ref{1.1}). In particular in \cite{J} Jannelli point out the $N=3$ is a critical dimension for (\ref{1.1}), while for  (\ref{1.20}) it is only a matter of how $\mu$ close to $\bar \mu$. Theorem \ref{b} in our paper is an extension of the results of \cite{J} (part a). Here $\varphi^{-1}(\frac{N-4s}{2}) =\bar \mu-1$ when $s=0$.

    	This paper is organized as follows.  In Section 2, we give some preliminary lemmas. In Section 3 we prove the main results  Theorem \ref{a} and Theorem \ref{b} by variational methods. In fact, we first prove Theorem \ref{a} by Mountain-Pass Theorem. Then we reduce Theorem \ref{b} to Theorem \ref{a} by choosing a test function to verify the assumption of Theorem \ref{a}.


\section{Preliminary lemmas}
In this section we give some preliminary lemmas. 
	\begin{lemma}\label{l2.2}
		Assume  $0<\mu<\Gamma_{N,s}$. Then there exist two positive constants $C$ and $c$ such that for all $u\in X_0$
		\begin{equation}\label{2.2}
		C\|u\|^2_{H^s(\R^N)}\leq c_{N,s}\int_{\R^{2N}}\frac{|u(x)-u(y)|^2}{|x-y|^{N+2s}}dxdy-\mu\int_{\Omega}\frac{|u(x)|^{2}}{|x|^{2s}}dx\leq c \|u\|^2_{H^s(\R^N)}.
		\end{equation}
		that is
		\begin{equation}\label{2.3jj}
		\|u\|_{{X_\mu}}=\Big(c_{N,s}\int_{\R^{2N}}\frac{|u(x)-u(y)|^2}{|x-y|^{N+2s}}dxdy-\mu\int_{\Omega}\frac{|u(x)|^{2}}{|x|^{2s}}dx\Bigl)^{1/2}
		\end{equation}
	\end{lemma}
	is a norm on $X_0$ equivalent to the usual one defined in (\ref{1.13j}).
	\begin{proof}
		On one hand, from (\ref{1.4j}), for all $0<\mu <\Gamma_{N,s}$,
		\begin{equation}\label{2.2j}
		\|u\|^2_{H^s(\R^N)}\geq c_{N,s}\int_{\R^{2N}}\frac{|u(x)-u(y)|^2}{|x-y|^{N+2s}}dxdy-\mu\int_{\Omega}\frac{|u(x)|^{2}}{|x|^{2s}}dx.\end{equation}
		On the other hand, 
		\begin{eqnarray}\label{2.3j}
		\begin{split}
		&c_{N,s}\int_{\R^{2N}}\frac{|u(x)-u(y)|^2}{|x-y|^{N+2s}}dxdy-\mu\int_{\Omega}\frac{|u(x)|^{2}}{|x|^{2s}}dx\\
		&\geq(1-\frac{\mu}{\Gamma_{N,s}})c_{N,s}\int_{\R^{2N}}\frac{|u(x)-u(y)|^2}{|x-y|^{N+2s}}dxdy
		\geq Cc_{N,s}\int_{\R^{2N}}\frac{|u(x)-u(y)|^2}{|x-y|^{N+2s}}dxdy
		\end{split}
		\end{eqnarray}
		and \begin{equation}\label{2.11j}
		\begin{split}
		\int_{\Omega}|u(x)|^2dx\leq |\Omega|^{(2^*_s-2)/2^*_s}\bigl(\int_{\Omega}|u|^{2^*_s}dx\bigl)^{\frac{2}{2^*_s}}&\leq c\int_{\R^{2N}}\frac{|u(x)-u(y)|^2}{|x-y|^{N+2s}}dxdy
		\end{split}
		\end{equation}
		then from (\ref{2.3j}) and (\ref{2.11j})
		\begin{equation}\label{2.6}
		\|u\|_{H^s(\R^N)}\leq c \bigl(c_{N,s}\int_{\R^{2N}}\frac{|u(x)-u(y)|^2}{|x-y|^{N+2s}}dxdy-\mu\int_{\Omega}\frac{|u(x)|^{2}}{|x|^{2s}}dx\bigl)^{1/2}.
		\end{equation}
		Collecting  with (\ref{2.2j}) and (\ref{2.6}), the proof is complete.
	\end{proof}		
%
%
\begin{lemma}\label{l2.1}
	Let $s\in (0,1), N>2s, \Omega$ be an open bounded set of $\R^N$ and $0<\mu<\Gamma_{N,s}$. Then 
	 problem (\ref{1.10j}) admits an eigenvalue $\lambda_{1,\mu}$ which is positive and that can be characterized as 
	\begin{equation}\label{2.1}
\dis	\lambda_{1,\mu}=\min_{u\in X_0\backslash \{0\}}\frac{c_{N,s}\int_{\R^{2N}}\frac{|u(x)-u(y)|^2}{|x-y|^{N+2s}}dxdy-  \mu\int_{\Om}\frac{|u(x)|^{2}}{|x|^{2s}}dx}{\int_{\Omega}|u(x)|^{2}dx}>0.
	\end{equation}
\end{lemma}
\begin{proof}It is remain to show that $\lambda_{1,\mu}$ can be attained and $\lambda_{1,\mu}>0$.
		Denote $$M=\{u\in X_0, \|u\|_{L^2(\Omega)}=1\},$$
		$$ I(u)=\frac{1}{2}\bigl (c_{N,s}\int_{\R^{2N}}\frac{|u(x)-u(y)|^2}{|x-y|^{N+2s}}dxdy-\mu\int_{\Om}\frac{|u(x)|^{2}}{|x|^{2s}}dx\bigl)=\frac{1}{2}\|u\|^2_{{X_\mu}}.$$
	Obviously (\ref{2.1}) is equivalent to 
		\begin{equation}\label{2.1j}
		\dis	\lambda_{1,\mu}=\min_{u\in M}(c_{N,s}\int_{\R^{2N}}\frac{|u(x)-u(y)|^2}{|x-y|^{N+2s}}dxdy-  \mu\int_{\Om}\frac{|u(x)|^{2}}{|x|^{2s}}dx)>0.
		\end{equation}

Let us take a minizing sequence $u_n\subset M$ such that
\begin{equation}
I(u_n)\rightarrow \inf_{u\in M}I(u)\geq 0 \text{ as }n\rightarrow \infty.
\end{equation}
From  Lemma \ref{l2.2}, $\|u_n\|_{H^s(\R^N)}$ is bounded. Then there exists $u_0$ such that (up to a subsequence, still denoted by $u_n$)
\begin{eqnarray}
u_n \rightharpoonup u_0 \text{ in } H^s(\R^N),
u_n \rightarrow u_0 \text{ in } L_{\text{loc}}^2(\R^N), \text{ and } u_n \rightarrow u_0 \text{ a.e. in }\R^N \text{ as } n\rightarrow \infty.
\end{eqnarray}
So $\|u_0\|_{L^2(\Omega)}=1\text{ and } u_0=0 \text{ a.e. in } \R^N\backslash \Omega$, that is $u_0 \in M$. Denote $v_n=u_n-u_0$, we have
\begin{eqnarray}
v_n \rightharpoonup 0 \text{ in } H^s(\R^N),
v_n \rightarrow 0 \text{ in } L_{\text{loc}}^2(\R^N), \text{ and } v_n \rightarrow 0 \text{ a.e. in }\R^N \text{ as } n\rightarrow \infty,
\end{eqnarray}
then by Br\'ezis-Lieb Lemma in \cite{BL}, it follows
\begin{equation}
\int_{\R^{2N}}\frac{|u_n(x)-u_n(y)|^2}{|x-y|^{N+2s}}dxdy=\int_{\R^{2N}}\frac{|v_n(x)-v_n(y)|^2}{|x-y|^{N+2s}}dxdy+\int_{\R^{2N}}\frac{|u_0(x)-u_0(y)|^2}{|x-y|^{N+2s}}dxdy+o_n(1),
\end{equation}
\begin{equation}
\int_{\Om}\frac{|u_n(x)|^{2}}{|x|^{2s}}dx=\int_{\Om}\frac{|v_n(x)|^{2}}{|x|^{2s}}dx+\int_{\Om}\frac{|u_0(x)|^{2}}{|x|^{2s}}dx+o_n(1),
\end{equation}
which implies that
\begin{equation}
\begin{split}
I(u_n)&=\frac{1}{2}\bigl (c_{N,s}\int_{\R^{2N}}\frac{|u_n(x)-u_n(y)|^2}{|x-y|^{N+2s}}dxdy-\mu\int_{\Om}\frac{|u_n(x)|^{2}}{|x|^{2s}}dx\bigl)\\&
=\frac{1}{2}\|v_n\|^2_{{X_\mu}}+\frac{1}{2}\|u_0\|^2_{X_\mu}+o_n(1)\\
&\geq \inf_{u\in M}I(u)(1+\|v_n\|^2_{L^2(\Omega)})
\end{split}
\end{equation}
where the last inequality follows from (\ref{2.1}) and (\ref{2.1j}).
Then
\begin{equation}
\inf_{u\in M}I(u)
=\lim_{n\rightarrow \infty}\frac{1}{2}\|v_n\|_{{X_\mu}}+\frac{1}{2}\|u_0\|^2_{X_\mu}
\geq \inf_{u\in M}I(u).
\end{equation}
Thus 
\begin{equation}
	 \inf_{u\in M}I(u)=I(u_0),
\end{equation}
that is, $u_0$ is a minimizer of $I(u).$ Since $\|u_0\|_{L^2(\Omega)}=1$, it is easy to obtain that $\lambda_{1,\mu}=I(u_0)>0$. The proof is complete.
\end{proof}

\section{The proof of main results}
In this section, we study the critical points of $J_\lambda$ which are solutions of problem (\ref{1.4}). In order to find these critical points, we will apply a variant of Mountain-Pass Theorem without the Palais-Smale condition (refer \cite{BN}). Due to the lack of compactness in the embedding $X_0\hookrightarrow L^{2^*_s}(\Omega)$, the functional $J_\lambda$ does not verify the Palais-Smale condition globally, but only in the energy range determined by the best fractional critical Sobolev constant $S$ given in (\ref{1.15}). 

First of all, we prove the functional $J_\lambda$ has the Mountain-Pass geometric structure.

\begin{proposition}\label{p3.1}
	Let $\lambda\in (0,\lambda_{1,\mu})$ and $f$ be a Carath\'edory function satisfying conditions (\ref{1.9})-(\ref{1.11}). Then there exists
	 $\rho>0$, $\beta>0$, $e\in X_0$  such that 
	 
	 1) for any $u\in X_0$ with $\|u\|_{X_\mu}=\rho$, $J_\lambda (u)\geq \beta$;
	 
	 2) $e\geq 0$ a.e. in $\R^N, \|e\|_{X_\mu}>\rho$ and $J_\lambda(e)<\beta$.
\end{proposition}
\begin{proof}
	For all $u\in X_0$, by (\ref{2.1}) and (\ref{1.12j}) we get for $\va>0$
	\begin{align*}
		J_{\lambda}&\geq \frac{1}{2}\Bigl (c_{N,s}\int_{\R^{2N}}\frac{|u(x)-u(y)|^2}{|x-y|^{N+2s}}dxdy-\lambda\int_{\Omega}|u(x)|^2dx-\mu\int_{\Om}\frac{|u(x)|^2}{|x|^{2s}}dx\Bigl )\\&\,\,\hspace{3mm}-\frac{1}{2^*_s}\int_{\Omega}|u(x)|^{2^*_s}dx-\va\int_{\Omega}|u|^2dx-\delta(\va)\int_{\Omega}|u(x)|^{2^*_s}dx\\
		&\geq \frac{1}{2}(1-\frac{\lambda}{\lambda_{1,\mu}}) \Bigl (c_{N,s}\int_{\R^{2N}}\frac{|u(x)-u(y)|^2}{|x-y|^{N+2s}}dxdy-\mu\int_{\Om}\frac{|u(x)|^2}{|x|^{2s}}dx\Bigl )\\&\,\,\hspace{3mm}-\va|\Omega|^{\frac{2^*_s-2}{2^*_s}}
		(\int_{\Omega}|u|^{2^*_s}dx)^{2^*_s/2}-(\delta(\va)+\frac{1}{2^*_s})\int_{\Omega}|u(x)|^{2^*_s}dx\\
	&	\geq \bigl(\frac{1}{2}(1-\frac{\lambda}{\lambda_{1,\mu}})-\va|\Omega|^{\frac{2^*_s-2}{2^*_s}}S\bigl) |u\|_{{X_\mu}}^{2}
-(\delta(\va)+\frac{1}{2^*_s})S^{2^*_s/2}\|u\|_{{X_\mu}}^{2^*_s/2}.
	\end{align*}
	Choose $\va>0$ small enough, there exist $\alpha>0$ and $\nu>0$ such that
	\begin{align}\label{3.1}
	J_{\lambda}(u)\geq \alpha \|u\|_{X_\mu}^2-\nu \|u\|_{X_\mu}^{2^*_s}.
	\end{align}
Now let $u\in X_0, \|u\|_{X_\mu}=\rho>0$, since $2^*_s>2$, choose $\rho$ small enough, then there exists $\beta>0$ such that
\begin{equation}
\inf_{u\in X_0,\|u\|_{{X_\mu}}=\rho}J_{\lambda}\geq \beta>0.
\end{equation}	
Fix any $u_0\in X_0$ with $\|u_0\|_{L^{2^*_s}(\Omega)}>0$. Without loss of generality,  we can assume $u_0\geq 0$ a.e. in $\R^N$ (otherwise, we can replace any $u_0\in X_0$ with its positive part, which belong to $X_0$ too, thanks to Lemma 5.2 in \cite{NPV}).
Since for $t>0$ large enough
\begin{equation}
F(x,tu_0)\geq-\frac{t^{2^{*}_s}}{22^*_s}|u_0|^{{2^*_s}},
\end{equation}
it follows that
\begin{equation}
J_{\lambda}(u_0)\leq ct^2 \|u_0\|^2_{X_\mu}-\frac{t^{2^*_s}}{22^*_s}\|u_0\|_{L^{2^*_s}}^{{2^*_s}}\rightarrow -\infty \text{ as } t\rightarrow \infty.
\end{equation}
Take 
\begin{equation}
e=t_0u_0
\end{equation}
 and $t_0$ sufficiently large, we have $J(t_0u_0)<0<\beta$.
The proof is complete.
\end{proof}
As a consequence of Proposition \ref{p3.1},  $J_{\lambda}(u)$ satisfies the geometry structure of Mountain-Pass Theorem.
Set 
\begin{equation}\label{3.5}
 c^*=:\inf_{\gamma \in \bar\Gamma} \sup_{t
	\in [0,1]}J_{\lambda}(\gamma (t)),
\end{equation}
 where $ \bar\Gamma=\{\gamma \in
C([0,1],X_0):\gamma(0)=0,\gamma(1)=e\in X_0\}$.

{\bf Proof of Theorem \ref{a} }
for all $\gamma\in \bar\Gamma $,  the function $t\rightarrowtail \|\gamma(t)\|_{X_\mu}$ is continuous, then there exists $\bar t\in (0,t)$ such that 
$\|\gamma (\bar t)\|_{X_\mu}=\rho$. It follows that
\begin{equation}
\sup_{t\in [0,1]} J_{\lambda}(\gamma(t))\geq  J_{\lambda}(\gamma (\bar t))\geq \inf_{u\in X_0, \|u\|_{X_\mu}=\rho}J_\lambda(u)\geq \beta,
\end{equation}
then $c^*\geq \beta$.
Since $te\in \bar\Gamma$, we have that
\begin{equation}\label{3.8}
\begin{split}
0<	\beta\leq c^*&=\inf_{\gamma \in \bar\Gamma} \sup_{t
 \in [0,1]}J_{\lambda}(\gamma (t))\\& \leq 
\sup_{t\in [0,1]}J_\lambda (te)\\&
\leq  \sup_{\zeta
 \geq 0}J_{\lambda}
(\zeta u_0)<\frac{s}{N}S^{\frac{N}{2s}}.
\end{split}
\end{equation}
By  Theorem 2.2 in \cite{BN} and Proposition \ref{p3.1}, there exists a Palais-Smale sequence $\{u_n\}\subset X_0$ a.e.,
\begin{equation}
J_{\lambda}(u_n)\rightarrow c^*, \,\,\,J'_{\lambda}(u_n)\rightarrow 0.
\end{equation}
From (\ref{1.12j}) and the definition of $J_\lambda$, it follows
\begin{eqnarray}\label{3.10}
\begin{split}
c^*+1&\geq  J_\lambda(u_n)-\frac{1}{2}<J'_\lambda(u_n),u_n>\\
      &=(\frac{1}{2}-\frac{1}{2^*_s})\|u_n\|^{2^*_s}_{L^{2^*_s}(\Omega)}+\int_{\Omega}
      \bigl (\frac{1}{2}f(x,u_n)u_n-F(x,u_n)\bigl )dx
      \\&\geq (\frac{1}{2}-\frac{1}{2^*_s})\|u_n\|^{2^*_s}_{L^{2^*_s}(\Omega)}-\va\int_{\Omega}|u_n|^{2^*_s} dx-c\int_{\Omega}|u_n|^{} dx\\
      &\geq c\|u_n\|^{2^*_s}_{L^{2^*_s}(\Omega)}-2\va \|u_n\|^{2^*_s}_{L^{2^*_s}(\Omega)}-C
      \end{split}
\end{eqnarray}
which implies $\|u_n\|^{2^*_s}_{L^{2^*_s}(\Omega)}$ is bounded for $\va>0$ small enough. So
\begin{eqnarray}\label{3.10}
\begin{split}
c^*+1&\geq
J_\lambda(u_n)\geq\frac{1}{2}(1-\frac{\lambda}{\lambda_{1,\mu}})\|u_n\|_{{X_\mu}}^2-\frac{1}{2^*_s}\|u_n\|^{2^*_s}_{L^{2^*_s}(\Omega)}-\int_{\Omega}F(x,u_n)dx
\\&\geq \frac{1}{2}(1-\frac{\lambda}{\lambda_{1,\mu}})\|u_n\|_{{X_\mu}}^2-\frac{1}{2^*_s}\|u_n\|^{2^*_s}_{L^{2^*_s}(\Omega)}-\va\|u_n\|^2_{L^2(\Om)}-C\|u_n\|^{2^*_s}_{L^{2^*_s}(\Omega)}
\\&\geq \frac{1}{2}(1-\frac{\lambda+2\va}{\lambda_{1,\mu}})\|u_n\|_{{X_\mu}}^2-c.
\end{split}
\end{eqnarray}
Thus $\|u_n\|_{{X_\mu}}$ is bounded for $\va>0$ small enough.

 Next we claim that there exists a solution $u_0\in X_0$ of (\ref{1.4}). Since $u_n$ is bounded in $X_0$, by Lemma {\ref{l2.2}}, up to a subsequence, still denoted by $u_n$, as $n\rightarrow \infty$
\begin{eqnarray}
u_n \rightharpoonup u_0 \text{ in } H^s(\R^N),\label {3.12jj}\\
u_n \rightarrow u_0 \text{ in } L_{\text{loc}}^p(\Omega),\text{ for all } p\in [1,2^*_s),\label {3.13j}\\
u_n\rightarrow u_0 \text{  a.e. in }\R^N.\label {3.14j}
\end{eqnarray}
Thus $u_0=0\text{ in }\R^N\backslash \Omega$, $u_0\in X_0$ and for all $\phi\in C^\infty_0(\Omega)$, as $n\rightarrow \infty$
\begin{equation}
\begin{split}
c_{N,s}\int_{\R^{2N}}\frac{(u_n(x)-u_n(y))(\phi(x)-\phi(y))}{|x-y|^{N+2s}}dxdy&\rightarrow c_{N,s}\int_{\R^{2N}}\frac{(u_0(x)-u_0(y))(\phi(x)-\phi(y))}{|x-y|^{N+2s}}dxdy,
\end{split}
\end{equation}
\begin{equation}
\begin{split}
\int_{\Omega}\frac{u_n(x)\phi(x)}{|x|^{2s}}dx\rightarrow \int_{\Omega}\frac{u_0(x)\phi(x)}{|x|^{2s}}dx, 
\end{split}
\end{equation}
\begin{equation}
\begin{split}
\int_{\Omega}u_n(x)\phi(x)dx\rightarrow \int_{\Omega}u_0(x)\phi(x)dx.
\end{split}
\end{equation}
Denote $g(u)=|u|^{2^*_s-2}u$.
From (\ref{3.12jj})-(\ref{3.14j}), (\ref{1.12j}) and $f(x,u),g(u)$ is continuous in $u$, applying Lemma A.2 in \cite{W}, 
\begin{equation}
\begin{split}
\|f(x,u_n)-f(x,u_0)\|_{L^1(\Omega)}\leq c\|u_n-u\|_{L^q(\Omega)}\rightarrow 0,\\
\|g(u_n)-g(u_0)\|_{L^1(\Omega)}\leq c\|u_n-u\|_{L^q(\Omega)}\rightarrow 0,\\
\end{split}
\end{equation}
where $q=2^*_s-1$.
Then it implies $\text{ for all } \phi\in C^\infty_0(\Omega)$, as $n\rightarrow \infty$
\begin{equation}
\begin{split}
|\int_{\Omega}(|u_n|^{2^*_s-2}u_n -|u_0|^{2^*_s-2}u_0)\phi dx|\leq c\|g(u_n)-g(u_0)\|_{L^1(\Omega)}&\rightarrow 0,
\end{split}
\end{equation}
\begin{equation}
\begin{split}
\int_{\Omega}(f(x,u_n) -f(x,u_0))\phi dx\leq c\|f(x,u_n)-f(x,u_0)\|_{L^1(\Omega)}&\rightarrow 0.\\
\end{split}
\end{equation}
Thus $u_0\in X_0$ is a solution of (\ref{1.4}). To complete the proof of Theorem 1.1, it suffices to show that  $u_0\not\equiv0$.

Suppose, in the contrary, $u_0\equiv 0$.
From (\ref{3.12jj})-(\ref{3.14j}) and (\ref{1.12j})
\begin{equation}
\begin{split}
\limsup_{n\rightarrow \infty}|\int_{\Omega}f(x,u_n)u_ndx|&\leq \va \limsup_{n\rightarrow \infty}\int_{\Om}|u_n|^{2^*_s}dx+c\limsup_{n\rightarrow \infty}\int_{\Om}|u_n|dx\leq c\va 
\end{split}
\end{equation}
where $\va>0$ is an arbitrary constant. Letting $\va\rightarrow 0$, it follows that
\begin{equation}\label{3.21j}
\begin{split}
\limsup_{n\rightarrow \infty}	|\int_{\Omega}f(x,u_n)u_ndx|\rightarrow 0.
\end{split}
\end{equation}
Similar as （(\ref{3.21j}), it follows
\begin{equation}\label{3.22j}
\begin{split}
\limsup_{n\rightarrow \infty}	|\int_{\Omega}F(x,u_n)dx|\rightarrow 0.
\end{split}
\end{equation}
Since 
\begin{equation}
\begin{aligned}\label{1.5j}
<J_{\lambda}'(u_n),u_n>&=c_{N,s}\int_{\R^{2N}}\frac{|u_n(x)-u_n(y)|^2}{|x-y|^{N+2s}}dxdy-\mu\int_{\Omega}\frac{|u_n(x)|^2}{|x|^{2s}}dx-\lambda\int_{\Omega}|u_n(x)|^2dx
\\&\ \ \ \ -\int_{\Omega}|u_n(x)|^{2^*_s}dx-\int_{\Omega}f(x,u_n)u_ndx,
\end{aligned}
\end{equation}
thanks to (\ref{3.21j}) and (\ref{3.22j}), we have
\begin{equation}
\begin{aligned}\label{1.5j}
c_{N,s}\int_{\R^{2N}}\frac{|u_n(x)-u_n(y)|^2}{|x-y|^{N+2s}}dxdy-\mu\int_{\Omega}\frac{|u_n(x)|^2}{|x|^{2s}}dx-\int_{\Omega}|u_n(x)|^{2^*_s}dx\rightarrow 0.
\end{aligned}
\end{equation}
Denote \[c_{N,s}\int_{\R^{2N}}\frac{|u_n(x)-u_n(y)|^2}{|x-y|^{N+2s}}dxdy-\mu\int_{\Omega}\frac{|u_n(x)|^2}{|x|^{2s}}dx\rightarrow L\]
which implies  
\[\int_{\Omega}|u_n(x)|^{2^*_s}dx\rightarrow L,\]
and 
\begin{equation}
0<\beta\leq c^*=\lim_{n\rightarrow\infty }J_\lambda{(u_n)}=(\frac{1}{2}-\frac{1}{2^*_s})L.
\end{equation}
It is easy to see that $L>0$. From (\ref{1.15})， we have
\begin{equation}
L\geq S L^ {\frac{2}{2^*_s}}，
\end{equation}
then
\begin{equation}
c^*\geq \frac{s}{N}S^{\frac{N}{2s}}
\end{equation}
which contradicts (\ref{3.8}). Hence $u_0\not\equiv 0$ in $\Omega$.
The proof of Theorem \ref{a} is complete.

For the proof of Theorem \ref{b}, the key step is to check the condition (\ref{1.20jj}). Firstly we choose the test function $v_\va(x)$ of (\ref{1.20jj}) as the following.
Let \begin{equation}
u(x)=\frac{U(x)}{\|U(x\|_{L^{2^*_s}(\R^N)}}, u_\va(x)=\va^{-\frac{N-2s}{2}}u(\frac{x}{\va}).
\end{equation}
Thus $\|u_\va\|_{L^{2^*_s}(\R^N)}=1$, and $u_\va(x)$ is the achieve function of $S$ (defined in (\ref{1.15})).
 Let us fix $\delta>0$ such that $B_{4\delta}\subset \Omega$, and let $\eta \in C^\infty_0(B_{2\delta})$ be such that $0\leq \eta\leq 1$ in $\R^N$, $\eta\equiv 1 \text{ in } B_{\delta}$.
For every $\va>0$, define
\begin{equation}
v_\va(x)=u_\va(x)\eta(x),x\in \R^N.
\end{equation}
Obviously $v_\va\in X_0$.   Secondly we make some estimates for $v_\va$.  The main strategy of the estimations for $v_\va(x)$ is similar to \cite{SV}.
\begin{proposition}\label{l3.1}
		1) 	For all $x\in B_r^c$, it follows
	\begin{eqnarray}
	|v_\va|\leq |u_\va|\leq c_r\va^{\frac{(N-2s)\eta_\mu}{2}},\label{3.13}\\
	|\nabla v_\va|\leq c_r\va^{\frac{(N-2s)\eta_\mu}{2}}\label{3.14}
	\end{eqnarray}
		where $r>0$ and $c_r$ is a positive constant depending on $r$.
	\item 2) For all $y\in B_\delta^C, x\in \R^N$ with $|x-y|< \delta/2$, it follows
	\begin{equation}\label{3.15j}
	|v_\va(x)-v_\va(y)|\leq c|x-y|\va^{\frac{(N-2s)\eta_\mu}{2}};
	\end{equation}
	for	all $x,y\in B_\delta^C$, it follows
	\begin{equation}\label{3.16j}
	|v_\va(x)-v_\va(y)|\leq c\min\{|x-y|,1\}\va^{\frac{(N-2s)\eta_\mu}{2}}.
	\end{equation}
	\end{proposition}

	\begin{proof}
		1)From the definition of $v_\va(x)$ and $u_{\va}(x)$, for $x\in B_r^c$, we have
		\begin{equation}\label{3.15}
		\begin{split}
		|v_\va(x)|\leq |u_\va(x)|&\leq 	\frac{c\va^{-(N-2s)/2}}{\Bigl(|x/\va|^{1-\eta_\mu}(1+|x/\va|^{2\eta_\mu})\Bigl )^{\frac{N-2s}{2}}}\\&=\frac{c\va^{\eta_\mu(N-2s)/2}}{\Bigl(|x|^{1-\eta_\mu}(\va^{2\eta_\mu}+|x|^{2\eta_\mu})\Bigl )^{\frac{N-2s}{2}}}\\&
		\leq c\va^{\eta_\mu(N-2s)/2} r^{-(1+\eta_\mu)(N-2s)/2}	
		\leq c_r\va^{\eta_\mu(N-2s)/2}
		\end{split}
		\end{equation}
		where $c_r$ is a positive constant depending on $r$.
		Since 
		\begin{equation}\label{3.16}
			\begin{split}
			|\nabla u_\va|&\leq c\va^{-(N-2s)/2}\bigl (|\frac{x}{\va}|^{1-\eta_\mu}(1+|\frac{x}{\va}|^{2\eta_\mu})\big)^{-\frac{N-2s}{2}-1}
			(|\frac{x}{\va}|^{\eta_\mu}+|\frac{x}{\va}|^{-\eta_\mu})\frac{1}{\va}\\
			&= c\va^{-(N-2s)/2}\bigl (|\frac{x}{\va}|^{1-\eta_\mu}(1+|\frac{x}{\va}|^{2\eta_\mu})\big)^{-\frac{N-2s}{2}}\frac{
			(|\frac{x}{\va}|^{\eta_\mu}+|\frac{x}{\va}|^{-\eta_\mu})}{|\frac{x}{\va}|^{1-\eta_\mu}(1+|\frac{x}{\va}|^{2\eta_\mu})}\frac{1}{\va}\\&
		\leq c |u_\va(x)|\frac{
			(|\frac{x}{\va}|^{\eta_\mu}+|\frac{x}{\va}|^{-\eta_\mu})}{|\frac{x}{\va}|^{1-\eta_\mu}(1+|\frac{x}{\va}|^{2\eta_\mu})}\frac{1}{\va}\leq  c |u_\va(x)|(\frac{1}{|x|}+\frac{\va^{2\eta_\mu}}{|x|^{2\eta_\mu+1}})\leq c_r |u_\va(x)|
			\end{split}
		\end{equation}
		and 
		\begin{equation}\label{3.17}
		\begin{split}
		|\nabla v_\va(x)|=|\nabla \eta(x)u_\va(x)+\nabla u_\va(x)\eta(x)|\leq c(|\nabla u_\va(x)|+|u_\va(x)|),
		\end{split}
		\end{equation}
		collecting (\ref{3.15})-(\ref{3.17}), we obtain (\ref{3.13}) and (\ref{3.14}).
		
		2)  Similar to Claim 10 in \cite{SV}, it follows (\ref{3.15j})-(\ref{3.16j}).
		
		 In fact,
		for all $x\in \R^N, y\in B_\delta^C, |x-y|<\frac{\delta}{2}$,
	      \begin{equation}
	      \begin{split}
	     |v_\va(x)-v_\va(y)|=|\nabla v_\va(\xi)||x-y|
	      \end{split}
	      \end{equation}
		where $\xi\in \R^N$  with $\xi=tx+(1-t)y$ for some $ t\in (0,1)$. Since
		\begin{equation}
		|\xi|=|y+t(x-y)|\geq |y|-t|x-y|> \frac{\delta}{2},
		\end{equation}
	taking $r=\frac{\delta}{2}$, from (\ref{3.14}) it follows (\ref{3.15j}).
	
	For all $x\in B_\delta^C , y\in B_\delta^C,$
	 if $ |x-y|\geq \frac{\delta}{2}$,
	\[|v_{\va}(x)-v_\va(y)|\leq |v_{\va}(x)|+|v_\va(y)|\leq  c \min\{1,|x-y|\} \va^{\eta_\mu(N-2s)/2}; \]
	if $ |x-y|\leq \frac{\delta}{2}$, (\ref{3.16j}) follows from (\ref{3.15j}).
	\end{proof}
\begin{proposition}\label{l3.2}
	Let $s\in (0,1)$, and $N>2s$, then the following estimates holds
	\begin{align}
	&1)\ \ \ c_{N,s}\int_{\R^{2N}}\frac{|v_\va(x)-v_\va(y)|^2}{|x-y|^{N+2s}}dxdy-\mu\int_{\Omega}\frac{|v_\va(x)|^2}{|x|^{2s}}dx\leq S+O(\va^{(N-2s)\eta_\mu});\label{3.23}\\
&	2)\,\, \|v_\va(x)\|^{2^*_s}_{L^{2^*_s}(\Omega)}= 1+O(\va^{\eta_\mu N});\label{3.24}
\\
&	3) \int_{\Omega}{|v_\va(x)|^2}dx\geq \begin{cases}
	c \va^{\eta_\mu(N-2s)}-c\va^{2s},& \text { if } \eta_\mu <\frac{2s}{N-2s},\\
	c|\log\va|\va^{2s}, &\text { if } \eta_\mu =\frac{2s}{N-2s},\\ -c\va^{\eta_\mu(N-2s)}+c\va^{2s},& \text { if } \eta_\mu >\frac{2s}{N-2s}.
	 	\end{cases}\label{3.25}
	\end{align}

\end{proposition}
\begin{proof}
1)	Define \begin{equation}
	D=\{(x,y)\in {\R^N}\times\R^N|x\in B_\delta, y\in B_\delta^C, |x-y|\geq \frac{\delta}{2}\}
	\end{equation}
	\begin{equation}
	E=\{(x,y)\in \R^N\times\R^N|x\in B_\delta, y\in B_\delta^C, |x-y|< \frac{\delta}{2}\}.
	\end{equation}
	Then 
	\begin{equation}\label{3.25j}
	\begin{split}
&c_{N,s}\int_{\R^{2N}}\frac{|v_\va(x)-v_\va(y)|^2}{|x-y|^{N+2s}}dxdy
\\&=c_{N,s}\int_{B_\delta\times B_\delta}\frac{|u_\va(x)-u_\va(y)|^2}{|x-y|^{N+2s}}dxdy+c_{N,s}\int_{D}\frac{|u_\va(x)-v_\va(y)|^2}{|x-y|^{N+2s}}dxdy\\
&\hspace{5mm}+c_{N,s}\int_{E}\frac{|u_\va(x)-v_\va(y)|^2}{|x-y|^{N+2s}}dxdy+c_{N,s}\int_{B_\delta^C\times B_\delta^C}\frac{|v_\va(x)-v_\va(y)|^2}{|x-y|^{N+2s}}dxdy.
	\end{split}
	\end{equation}

From (\ref{3.15j}) and $E\subset B_\delta\times  B_{2\delta}$ , we have
\begin{equation}\label{3.45}
\begin{split}
c_{N,s}\int_{E}\frac{|u_\va(x)-v_\va(y)|^2}{|x-y|^{N+2s}}dxdy\leq
c_{N,s}\int_{B_\delta\times  B_{2\delta}}\frac{\va^{(N-2s)\eta_\mu}}{|x-y|^{N+2s-2}}dxdy\leq c_\delta \va^{(N-2s)\eta_\mu},
\end{split}
\end{equation}
where $c_\delta>0$ is a constant depending on $\delta$.
Similarly, from (\ref{3.16j}) 
\begin{equation}\label{3.46}
\begin{split}
	c_{N,s}\int_{B_\delta^C\times B_\delta^C}\frac{|u_\va(x)-v_\va(y)|^2}{|x-y|^{N+2s}}dxdy&\leq
	c_{N,s}\int_{B_\delta^C\times B_\delta^C}\frac{\va^{(N-2s)\eta_\mu}\min\{1,|x-y|^2\}}{|x-y|^{N+2s}}dxdy\\
	&\leq c_\delta\va^{(N-2s)\eta_\mu}.
\end{split}
\end{equation}
where $c_\delta>0$ is a constant depending on $\delta$.
Now, it remains to estimate the integral on $D$ of (\ref{3.25j}),
\begin{equation}\label{3.30}
\begin{split}
&c_{N,s}\int_{D}\frac{|u_\va(x)-v_\va(y)|^2}{|x-y|^{N+2s}}dxdy\\&
\leq c_{N,s}\int_{D}\frac{|u_\va(x)-u_\va(y)|^2}{|x-y|^{N+2s}}dxdy+c_{N,s}\int_{D}\frac{|u_\va(y)-v_\va(y)|^2}{|x-y|^{N+2s}}dxdy\\
&\hspace{5mm}+c_{N,s}\int_{D}\frac{2|v_\va(y)-u_\va(y)||u_\va(x)-u_\va(y)|}{|x-y|^{N+2s}}dxdy
\end{split}
\end{equation}
Now we estimate the last term in the right hand side of (\ref{3.30}).
Once more by (\ref{3.13}), \begin{equation}\label{3.31}
\begin{split}
&c_{N,s}\int_{D}\frac{|u_\va(y)-v_\va(y)|^2}{|x-y|^{N+2s}}dxdy\\&\leq
c_{N,s}\int_{D}\frac{|u_\va(y)|^2+|v_\va(y)|^2}{|x-y|^{N+2s}}dxdy\\
&\leq c\int_{|x-y|\geq \delta/2}\frac{\va^{(N-2s)\eta_\mu}}{|x-y|^{N+2s}}dxdy\leq c_{\delta}\va^{(N-2s)\eta_\mu}.
\end{split}
\end{equation}
where $c_\delta>0$ is a constant depending on $\delta$.

And  for $(x,y)\in D$,
\begin{equation}
\begin{split}
|u_\va(x)v_\va(y)|\leq |u_\va(x)u_\va(y)|&\leq c\va^{(N-2s)\eta_\mu/2}\frac{\va^{-(N-2s)/2}}{\Bigl(|x/\va|^{1-\eta_\mu}(1+|x/\va|^{2\eta_\mu})\Bigl )^{\frac{N-2s}{2}}}\\
&\leq \frac{c}{\Bigl(|x|^{1-\eta_\mu}(1+|x/\va|^{2\eta_\mu})\Bigl )^{\frac{N-2s}{2}}},
\end{split}
\end{equation} 
let $\bar x=\frac{x}{\va},z=y-x$, it follows
 \begin{equation}\label{3.33}
\begin{split}
&c_{N,s}\int_{D}\frac{|u_\va(x)u_\va(y)|}{|x-y|^{N+2s}}dxdy\\&\leq 
c\int_{D}\frac{1}{\Bigl(|x|^{1-\eta_\mu}(1+|x/\va|^{2\eta_\mu})\Bigl )^{\frac{N-2s}{2}}|x-y|^{N+2s}}dxdy \\&\leq 
c\va^N\int_{|\bar x|\leq \delta/\va, |z|>\frac{\delta}{2}}\frac{1}{\Bigl(|\va \bar  x|^{1-\eta_\mu}(1+|\bar x|^{2\eta_\mu})\Bigl )^{\frac{N-2s}{2}}|z|^{N+2s}}d\bar xdz\\
&\leq c\va^{N-\frac{(N-2s)(1-\eta_\mu)}{2}}\int_{|\bar x|\leq \delta/\va}\frac{1}{\Bigl(| \bar  x|^{1-\eta_\mu}(1+|\bar x|^{2\eta_\mu})\Bigl )^{\frac{N-2s}{2}}}d\bar x\\
&\leq  c\va^{N-\frac{(N-2s)(1-\eta_\mu)}{2}}\Bigl (\int_{|\bar x|\leq 1 }\frac{1}{\Bigl(| \bar  x|^{1-\eta_\mu}(1+|\bar x|^{2\eta_\mu})\Bigl )^{\frac{N-2s}{2}}}d\bar x+\int_{1\leq |\bar x|\leq\delta/\va}\frac{1}{\Bigl(| \bar  x|^{1-\eta_\mu}(1+|\bar x|^{2\eta_\mu})\Bigl )^{\frac{N-2s}{2}}}d\bar x\Bigl )\\
&\leq c\va^{N-\frac{(N-2s)(1-\eta_\mu)}{2}}\Bigl (\int_{r\leq 1 }\frac{r^{(N-1)-(1-\eta_\mu)\frac{N-2s}{2}}}{\Bigl(1+r^{2\eta_\mu}\Bigl )^{\frac{N-2s}{2}}}dr+\int_{1\leq r\leq\delta/\va}r^{(N-1)-\frac{(1+\eta_\mu)(N-2s)}{2}}dr\Bigl )\\
 &\leq c\va^{N-\frac{(N-2s)(1-\eta_\mu)}{2}}\Bigl (\int_{r\leq 1 }r^{(N-1)-(1-\eta_\mu)\frac{N-2s}{2}}dr+\va^{-N+\frac{(1+\eta_\mu)(N-2s)}{2}}+c\Bigl )\\
 &=O(\va^{(N-2s)\eta_\mu}).
\end{split}
\end{equation}	
 It follows from (\ref{3.30})-(\ref{3.33}) that
 \begin{equation}\label{3.30j}
 \begin{split}
 c_{N,s}\int_{D}\frac{|u_\va(x)-v_\va(y)|^2}{|x-y|^{N+2s}}dxdy
 \leq c_{N,s}\int_{D}\frac{|u_\va(x)-u_\va(y)|^2}{|x-y|^{N+2s}}dxdy+O(\va^{(N-2s)\eta_\mu}).
 \end{split}
 \end{equation}

Thus from (\ref{3.45}), (\ref{3.46}) and (\ref{3.30j}),
	\begin{equation}\label{3.34}
	\begin{split}
	c_{N,s}\int_{\R^{2N}}\frac{|v_\va(x)-v_\va(y)|^2}{|x-y|^{N+2s}}dxdy&\leq   c_{N,s}\int_{\R^{2N}}\frac{|u_\va(x)-u_\va(y)|^2}{|x-y|^{N+2s}}dxdy+O(\va^{(N-2s)\eta_\mu}).
	\end{split}
	\end{equation}
Since 	
	\begin{equation}
	\begin{split}
	|\mu\int_{\R^N}\frac{(1-\eta(x)^2)|u_\va(x)|^2}{|x|^{2s}}dx|&\geq c|\int_{|x|\geq 2\delta}\frac{|u_\va(x)|^2}{|x|^{2s}}dx|
	\\&= c|\int_{|x|\geq 2\delta}\frac{\va^{-(N-2s)}}{\Bigl(|x/\va|^{1-\eta_\mu}(1+|x/\va|^{2\eta_\mu})\Bigl )^{N-2s}|x|^{2s}} dx	\\
	&= c|\int_{|x|\geq 2\delta/\va}\frac{1}{\Bigl(|x|^{1-\eta_\mu}(1+|x|^{2\eta_\mu})\Bigl )^{N-2s}|x|^{2s}} dx\\&\geq
	 c|\int_{r\geq 2\delta/\va}r^{N-1-2s-(1+\eta_\mu){(N-2s)}} dr= O(\va^{(N-2s)\eta_\mu}),
	\end{split}
	\end{equation}
	and
	\begin{equation}
	\begin{split}
	|\mu\int_{\R^N}\frac{(1-\eta(x)^2)|u_\va(x)|^2}{|x|^{2s}}dx|&\leq c|\int_{|x|\geq \delta}\frac{|u_\va(x)|^2}{|x|^{2s}}dx|
	\\&\leq c|\int_{|x|\geq \delta}\frac{\va^{-(N-2s)}}{\Bigl(|x/\va|^{1-\eta_\mu}(1+|x/\va|^{2\eta_\mu})\Bigl )^{N-2s}|x|^{2s}} dx= O(\va^{(N-2s)\eta_\mu}),
	\end{split}
	\end{equation}
	it gives
	\begin{equation}\label{3.36}
	\begin{split}
	\mu\int_{\R^N}\frac{|v_\va(x)|^2}{|x|^{2s}}dx&
	=\mu\int_{\R^N}\frac{|u_\va(x)|^2}{|x|^{2s}}dx-\mu\int_{\R^N}\frac{(1-\eta(x)^2)|u_\va(x)|^2}{|x|^{2s}}dx
	\\&= \mu\int_{\R^N}\frac{|u_\va(x)|^2}{|x|^{2s}}dx+O(\va^{(N-2s)\eta_\mu}).
	\end{split}
	\end{equation}
		From (\ref{3.34}) and (\ref{3.36}), we have
		\begin{equation}\label{3.37}
		\begin{split}
		c_{N,s}\int_{\R^{2N}}\frac{|v_\va(x)-v_\va(y)|^2}{|x-y|^{N+2s}}dxdy-\mu\int_{\R^N}\frac{|v_\va(x)|^2}{|x|^{2s}}dx\leq  S+O(\va^{(N-2s)\eta_\mu}).
		\end{split}
		\end{equation}
		
		2) Similar to (\ref{3.36}), we have
		\begin{equation}
		\int_{\R^N}|v_\va(x)|^{2^*_s}dx= \int_{\R^N}|u_\va(x)|^{2^*_s}dx+O(\va^{N\eta_\mu})
		\end{equation}
	
		3)Finally, we want to prove (\ref{3.25})\begin{equation}
	\begin{split}
	\int_{\Omega}{|v_\va(x)|^2}dx&\geq 	\int_{|x|\leq \delta}{|u_\va(x)|^2}dx \\&\geq c\int_{|x|\leq \delta}\frac{\va^{-(N-2s)}}{\Bigl(|x/\va|^{1-\eta_\mu}(1+|x/\va|^{2\eta_\mu})\Bigl )^{N-2s}} dx	\\
	&= \va^{2s}c\int_{|x|\leq \delta/\va}\frac{1}{\Bigl(|x|^{1-\eta_\mu}(1+|x|^{2\eta_\mu})\Bigl )^{N-2s}} dx\\
	&\geq c\va^{2s}\int_1 ^{\delta/\va}\frac{r^{N-1}}{\Bigl(r^{1-\eta_\mu}(1+r^{2\eta_\mu})\Bigl )^{N-2s}} dr\\
	&=\begin{cases}
	c \va^{\eta_\mu(N-2s)}-c\va^{2s},& \text { if } \eta_\mu<\frac{2s}{N-2s},\\
	c|\log\va|\va^{2s}, &\text { if } \eta_\mu=\frac{2s}{N-2s},\\ -c\va^{\eta_\mu(N-2s)}+c\va^{2s},& \text { if } \eta_\mu>\frac{2s}{N-2s}.
	\end{cases}
	\end{split}
	\end{equation}	
\end{proof}
{\bf Proof of Theorem 1.2} From Theorem \ref{a}, we only need to verify (\ref{1.20jj}).
 Denote
 $t_\va$ be the attaining point of $\dis\max_{t>0}J_{\lambda}(tv_\va)$.
 Similar as the proof of  Lemma 8.1 in \cite{GY},  let
$t_\va$ be the attaining point of $\dis\max_{t>0}J_\lambda \left(tv_\va\right)$, we claim $t_\va$ is uniformly bounded for $\va>0$ small.  In fact,
we consider the function
\begin{equation}
\begin{split}
g(t)=J_\lambda \left(tv_\va\right)&=\frac{t^2}{2}(\|v_\va(x)\|^2_{X_\mu}-\int_{\Om}\la |v_\va|^2 dx)-\frac{t^{{2^*_s}}}{{2^*_s}}\int_\Om
|v_\va|^{{2^*_s}}dx
\\& \geq
\frac{t^2}{2}(1-\frac{\lambda}{\lambda_{1,\mu}})\|v_\va(x)\|^2_{X_\mu}-\frac{t^{{2^*_s}}}{{2^*_s}}\int_\Om
|v_\va|^{{2^*_s}}dx.
\end{split}
\end{equation} 
Since $\displaystyle\lim_{t\rightarrow +\infty} g(t)=-\infty$ and $g(t)>0$ when $t$ closed to $0$, so that $\displaystyle\max_{t>0}g(t)$ is attained for $t_\va>0$. Then
\begin{equation}\label{3.11j}
g'(t_\va)=t_\va(\|v_\va(x)\|^2_{X_\mu}-\lambda\int_{\Om} |v_\va|^2 dx)-{t_\va^{{2^*_s-1}}}\int_\Om
|v_\va|^{{2^*_s}}dx=0.\end{equation} 
 From  (\ref{3.11j}) and Lemma \ref{l3.2}, for $\va$ sufficiently small, 
we have \begin{equation}\label{3.12j}\frac{S}{2}(1-\frac{\lambda}{\lambda_\mu})\leq t_\va^{{2^*_s-2}}=\frac{\|v_\va(x)\|^2_{X_\mu}-\lambda\int_{\Om} |v_\va|^2 dx}{\int_\Om
|v_\va|^{{2^*_s}}dx}<2S.\end{equation}
which implies $t_\va$ is bounded for $\va>0$ small enough.


 From the definition of $\eta_\mu$, (\ref{1.19}) and (\ref{1.20jjj}), we have
  \begin{equation}
 \mu\leq \varphi^{-1}(\frac{N-4s}{4}) \Longleftrightarrow \eta_\mu \geq \frac{2s}{N-2s}.
  \end{equation}
 Hence  for
 $\va>0$ sufficient small and $\mu\leq \varphi^{-1}(\frac{N-4s}{4})$,
 \begin{align*}
 \max_{t>0}J_{\lambda}(tv_\va)&=J_{\lambda} (t_\va v_\va)\\&\leq
 \max_{t>0}\Bigl\{\frac{t^2}{2}\Bigl (c_{N,s}\int_{\R^{2N}}\frac{|v_\va(x)-v_\va(y)|^2}{|x-y|^{N+2s}}dxdy-\mu\int_{\Om}\frac{|v_\va(x)|^2}{|x|^{2s}}dx\Bigl )
 -\frac{t^{{2^*_{s}}}}{{2^*_{s}}}\int_{\Om}
 |v_\va(x)|^{{2^*_{s}}}dx\Bigl\}
 \\&\ \ \   
 -\begin{cases}
 c \va^{\eta_\mu(N-2s)}-c\va^{2s},& \text { if } \eta_\mu<\frac{2s}{N-2s},\\
 c|\log\va|\va^{2s}, &\text { if } \eta_\mu=\frac{2s}{N-2s},\\ -c\va^{\eta_\mu(N-2s)}+c\va^{2s},& \text { if } \eta_\mu>\frac{2s}{N-2s},
 \end{cases}\\
 &<\frac{s}{N}(S+\va^{\eta_\mu(N-2s)})^{\frac{N}{2s}}-\begin{cases}
 c \va^{\eta_\mu(N-2s)}-c\va^{2s},& \text { if } \eta_\mu<\frac{2s}{N-2s},\\
 c|\log\va|\va^{2s}, &\text { if } \eta_\mu=\frac{2s}{N-2s},\\ -c\va^{\eta_\mu(N-2s)}+c\va^{2s},& \text { if } \eta_\mu>\frac{2s}{N-2s},
 \end{cases}\\
 &<\frac{s}{N}S^{\frac{N}{2s}}.
 \end{align*}
 This completes the proof of (\ref{1.20jj}). 
 

\end{document}